\documentclass[a4paper,11pt]{article}
\usepackage{amssymb,amsthm,amsmath,amscd}
\usepackage{epsfig}
\usepackage{subfigure}
\usepackage{fancyhdr}
\usepackage[top=1in,bottom=1.5in,left=1in,right=1in]{geometry}
\usepackage[all]{xy}
\newtheorem{theorem}{Theorem}[section]

\newtheorem{lemma}{Lemma}[subsection]
\newtheorem{proposition}{Proposition}[section]
\newtheorem*{theorem*}{Theorem}

\newcommand{\half}{\frac{1}{2}}
\newcommand{\be}{\begin{eqnarray}}
\newcommand{\ee}{\end{eqnarray}}
\newcommand{\ben}{\begin{eqnarray*}}
\newcommand{\een}{\end{eqnarray*}}

\DeclareMathOperator{\dist}{dist}

\DeclareMathOperator{\supp}{supp}
\DeclareMathOperator{\diam}{diam}
\DeclareMathOperator{\pds}{{p_{\it d}}sin}
\DeclareMathOperator{\SCale}{{scale}}

\def\Supp{{\supp(\mu)}}
\def\eps{\epsilon}

\def\agh{{\mathrm{AG}_{d}(H)}}

\newcommand{\reals}{\mathbb R}

\newcommand{\nats}{\mathbb N}

\def\hmin{{c_{\mathrm{ht}}}}
\def\supp{{\mathrm{supp}}}

\def\Max0{{\mathrm{max}_{x_0}}}
\def\Min0{{\mathrm{min}_{x_0}}}
\def\cdes{{c_{\mathrm{Dsh}}}}
\def\cdesh2{{c^2_{\mathrm{Dsh}}}}
\def\cDn{{c_{\mathrm{vol,Dsh}}}}
\def\cDnn2{{c^2_{\mathrm{vol,Dsh}}}}
\def\cvolk{{c_{\mathrm{vol},\mu}}}
\def\calcvolk{{{\cal C}_{\mathrm{vol}}}}
\def\cvolker{{c^2_{\mathrm{vol},\mu}}}
\def\cvol{{c_{\mathrm{vol}}}}
\def\cpols{{c_{\mathrm{pol}}}}
\def\cpol2{{c^2_{\mathrm{pol}}}}
\def\cdlss{{c_{\mathrm{dls}}}}
\def\cdls2{{I_{\mathrm{dls}}}}
\def\cdls2{{c^2_{\mathrm{dls}}}}
\def\xcm{{x_{\mathrm{cm}}}}

\def\MM{{\mathrm{M}}}

\def\Tbe{{\mathrm{T_{ube}}}}
\newcommand{\di}{\, \mathrm{d}}
\newcommand{\fdi}{\mathrm{d}}
\newcommand{\latop}[2]{\substack{{#1}\\{#2}}} 

\providecommand{\perm}[2]{\mathrm{P}(#1,#2)}

\begin{document}

\title{Least Squares Approximations of Measures via Geometric Condition Numbers\thanks{This work has been supported
by NSF awards DMS-0612608, DMS-0915064 and DMS-0956072. The authors thank the reviewer for many helpful comments and suggestions and
Alin Pogan and Teng Zhang for some helpful discussions (references~\cite{halmos_book78} and~\cite{yafaev_book92} were suggested by Alin).}}
\author{Gilad Lerman\thanks{Contact:lerman@umn.edu, jonathan.t.whitehouse@vanderbilt.edu} \
and J. Tyler Whitehouse$\hspace{0.01cm}^{\dag}$
}

\maketitle
\begin{abstract}
For a probability measure $\mu$ on a real separable Hilbert space $H$, we are interested in ``volume-based'' approximations of the $d$-dimensional least squares error of $\mu$, i.e., least squares error with respect to a best fit $d$-dimensional affine subspace. Such approximations are given by averaging real-valued multivariate functions which are typically scalings of squared $(d+1)$-volumes of $(d+1)$-simplices in $H$.
Specifically, we show that such averages are comparable to the square of the
$d$-dimensional least squares error of $\mu$, where the comparison depends on a simple quantitative
geometric property of $\mu$.
This result is a higher dimensional generalization of the elementary
fact that the double integral of the squared distances between points is proportional to the variance of
$\mu$. We relate our work to two recent algorithms, one for clustering affine subspaces and the other for Monte-Carlo SVD based on volume sampling.

\end{abstract}
\section{Introduction}

Our setting includes a real
separable Hilbert space $H$ (with dot product $\langle \cdot,\cdot \rangle$ and induced
norm $\Vert \cdot \Vert$), a Borel probability measure $\mu$ on
$H$ and a fixed intrinsic dimension $d \in \nats$. We  assume
that the support of $\mu$ is bounded. Let $\agh$ denote the affine Grassmannian
on $H$, that is, the set of all $d$-flats (i.e., $d$-dimensional affine subspaces) in $H$.
The $d$-dimensional
least squares (LS) error for $\mu$ is
\be \label{eq:def_e2}
e_2(\mu,d)=\inf_{L \in \agh}\sqrt{\int\dist^2(x,L)\di\mu(x)},\ee
where $\dist(x,L)$ denotes the distance of $x\in H$ to $L$.

We form functions $c: H^{d+2} \rightarrow \reals$, whose integrals approximate $e^2_2(\mu,d)$.
Denoting an arbitrary element of $H^{d+2}$ by $X=(x_0,\ldots,x_{d+1})$ and viewing it as a $(d+1)$-simplex in $H$, we
express the desired comparison as follows.
\be \label{eq:typical_result} e^2_2(\mu,d) \approx
\int_{H^{d+2}}c^2(X)\di\mu^{d+2}(X)\ee
(i.e., the ratios of the LHS and RHS of~\eqref{eq:typical_result} are bounded by constants, which may depend on $\mu$).
Some of these functions are obtained by appropriate scaling of $(d+1)$-volumes.
We denote by $\MM_{d+1}(X)$ the $(d+1)$-volume of
any of the parallelotopes generated
by the vertices of $X$. We also denote the diameter of $X$ by $\diam(X)$, i.e., the maximal edge length.
An example of such  a function $c$ is obtained by scaling $\MM_{d+1}(X)$ by a power of the diameter, i.e.
\begin{equation}\label{volume-1}\cvol(X)=\frac{\MM_{d+1}(X)}{\diam^d(X)}.\end{equation}
We refer to such functions as geometric condition numbers (GCNs), since they measure the geometric conditioning of the simplex $X$ by
a quantity that scales like the diameter of the simplex. The smaller they are the flatter, i.e., better-conditioned, the simplex is.

When $d=0$, \eqref{eq:typical_result} reduces to an elementary though useful identity, which we exemplify for the GCN $\cvol$.
In this
case, the best approximating $0$-flat (i.e., best approximating point) is
the mean, $\int x \di \mu$, and $e_2^2(\mu,0)$
is the variance of $\mu$, that is,
$$
e_2^2(\mu,0) = \int \left\| x - \int x \di \mu(x) \right\|^2 \fdi \mu(x) \equiv \frac{1}{2}
\int \|x_1-x_2\|^2 \di \mu(x_1) \di \mu(x_2).
$$
Moreover, %
$$
\cvol(x_1,x_2) = \|x_1-x_2\|
$$
and consequently,
\be \label{eq:elementary_identity} e_2^2(\mu,0)
=\frac{1}{2}
\int \cvol^2(x_1,x_2) \di \mu(x_1) \di \mu(x_2). \ee
Since our GCNs (of $d+2$ variables) are constant multiples of the
pairwise distance when $d=0$, this identity extends to all of them (with possibly a different multiplicative constant).

This paper  generalizes~\eqref{eq:elementary_identity} to higher
dimensional approximations and obtains estimates like~\eqref{eq:typical_result} for various GCNs.
This generalization restricts the type of measure $\mu$ by various conditions (depending on the GCN).
Our weakest condition, which we refer to as $d$-separation tries to avoid the concentration of $\mu$ around a subspace of dimension lower than $d$
(see Section~\ref{sec:d_separated} for precise definition).

This investigation is partly motivated by the analysis of a recent spectral clustering method for data sampled from multiple subspaces~\cite{spectral_theory, spectral_applied}.
The goodness of clustering for this method depends on the averaged GCN within each cluster and the theory developed here interprets this dependence in terms of the $d$-dimensional LS errors within clusters. We also  relate our study to some aspects of volume-based sampling for fast SVD~\cite{DRVW06,DV06}.

Many of our techniques are rooted in the theory of uniform
rectifiability~\cite{DS91}. In particular, notions similar to the $d$-separation condition have appeared before for $d$-regular or upper $d$-regular measures (see Section~\ref{sec:lower_bound} for their definitions) in~\cite[Lemma 5.8]{DS91}, \cite[Lemma 2.3]{Leger99}, \cite[Lemma~8.2]{tolsa_riesz_rect} and~\cite[Proposition~3.1]{LW-part2}.
Moreover, differently scaled functions of $d+2$ variables, referred to as discrete curvatures, were studied in~\cite{MMV96, Leger99, LW-part1, LW-part2} for $d$-regular measures. For example, while $\MM_{d+1}(X)$ is scaled by $\diam^d(X)$ to produce the geometric condition number $\cvol$, it can be scaled differently to obtain the following discrete curvature:
\begin{equation}\label{curv_volume-2}\calcvolk(X)=\frac{\MM_{d+1}(X)}{\diam^{(d+1)^2}(X)}\,.\end{equation}
It follows from~\cite{LW-part1, LW-part2} that for $d$-regular measures the integral of $\calcvolk^2$ is comparable to the Jones-type flatness, which adds up appropriately normalized $d$-dimensional LS errors of certain balls of different radii centered at different locations.
Another type of scaling of $\MM_{d+1}$ (or more precisely, an equivalent variant of it) appeared in~\cite{Strzelecki:1221693} for exploring different geometric properties of the underlying measure.

\subsection{Structure of This Paper and Additional Results}

In Section~\ref{sec:context} we introduce notational
conventions. In Section~\ref{sec:ls_d_flat} we verify the existence of a LS
$d$-flat minimizing the error $e_2(\mu,d)$ and construct it in terms of
the singular value decomposition of a special operator,
which we refer to as the data-to-features operator.
In Section~\ref{sec:more-LS-curvatures} we introduce $d$-dimensional GCNs of $d+2$ variables, in addition to $\cvol$.
Section~\ref{sec:upper_bound} controls $e_2^2(\mu,d)$ from above by integrals of these GCNs, whereas
Section~\ref{sec:lower_bound} bounds $e_2^2(\mu,d)$ from below by these integrals and thus concludes the desired comparisons.
In Section~\ref{sec:gcn_reduce}, we form $d$-dimensional GCNs of both $d+1$ and $d$ variables, and we establish their comparisons.
We also relate there our work to
that of Deshpande et al.~\cite{DRVW06, DV06}.
Section~\ref{sec:stat} puts this work in a statistical context by relating our results to clustering affine subspaces as well as extending some of the previous comparisons with high probability to the corresponding empirical quantities estimated from i.i.d.~samples from $\mu$.
We discuss further implications and possible extensions in Section~\ref{sec:discuss}.

\section{Notational Conventions}
\label{sec:context}

\subsection{Comparisons}

For real-valued functions $f$ and $g$, we let $f \lessapprox g$ denote the existence of  $C>0$ such that $f \leq C \cdot g$.
Similarly, $f \approx g$ if $f
\lessapprox g$ and $g \lessapprox f$. The constants
may depend on some arguments of $f$ and $g$, which we indicate if
they are unclear from the context.

\subsection{Simplices}

Fixing $n \in \nats$, $n \geq 2$, we represent
$n$-simplices in $H$ by
ordered $(n+1)$-tuples of the product space,
$H^{n+1}$. We denote an element of
$H^{n+1}$ by $X=(x_0,\ldots,x_{n})$ and for $0 \leq i \leq n$:  $(X)_i=x_i$ denotes the projection of $X$ onto its
$i^{\text{th}}$ $H$-valued coordinate (or vertex).
For $0\leq i<j\leq n$, $y,z\in H$ and $X \in H^{n+1}$ as above, we form the following elements:
\begin{equation}\label{eq:def-removal1}
X(i)=(x_0,\ldots, x_{i-1},x_{i+1},\ldots,x_{n}),
\end{equation}
\begin{equation}\label{eq:def-removal2}
X(y,i)=(x_0,\ldots,x_{i-1},y,x_{i+1},\ldots,x_{n}),
\end{equation}

The  minimal edge
length of $X$ is denoted by $\min(X)$.
We define the following quantities of $X$ with respect to its
zeroth coordinate $x_0$:
\begin{equation}\Max0(X)=\max_{1\leq j\leq n}\|x_j-x_0\|\ \textup{
and }\ \Min0(X)=\min_{1\leq j\leq n}\|x_j-x_0\|.\end{equation}
For $X$ such that $\min(X)\not=0$, let
\be\SCale_{x_0}(X)=\frac{\Min0(X)}{\Max0(X)}.\label{eq:def_scale}\ee
We say that a simplex $X$ is well-scaled at $x_0$ (for $\lambda>0$)
if $\min(X)>0$ and $\SCale_{x_0}(X)\geq \lambda$.

We let $L[X]$ denote the affine subspace
of $H$ of minimal dimension containing the vertices of $X$.
We recall that for $n \in \nats$, $\MM_{n}(X)$ is the $n$-volume of any of
the parallelotopes generated by the vertices of $X$.
We note that
\begin{equation}\label{product-formula}\MM_{n}(X)=\dist(x_i,L[X(i)])\cdot\MM_{n-1}(X(i)) \ \text{ for all }
0\leq i\leq d+1.\end{equation}

\section{Least Squares $d$-Flats and Their Construction}
\label{sec:ls_d_flat}

Formally, a LS $d$-flat for $\mu$ is a $d$-flat $L\in \agh$, for which the RHS of~\eqref{eq:def_e2}
obtains its minimal value. We show here that such $d$-flats exist, i.e., the function
\be\label{eq:def_FL} F(L) = \int\mbox{dist}^2(x,L)\di \mu (x) \,  \ee
obtains its minimum among all $d$-flats $L$ in $\agh$. Moreover we show how
to construct a LS $d$-flat given the singular value decomposition (SVD) of the data-to-features operator
described next.

\subsection{The Data-to-Features Operator}
\label{subsec:data_to_feature}
We define the center of mass of $\mu$, $\xcm$, by
\be
\label{eq:def_cm}
\xcm \equiv \xcm(\mu)=\int x\di \mu(x)
\ee
and denote by $L_2(\mu)$ the set of functions $f:H \to \reals$ such that
$\int | f(x)|^2 \di \mu(x) < \infty$.
The data-to-features operator $A_\mu :H \to L_2(\mu)$ is
\be \label{eq:data_to_feature} (A_\mu y)(x) = \langle y,x-\xcm \rangle \ \mbox{for
all}\ x,y \in H . \ee
We use the name ``data-to-features'' operator since if $\mu$ is an
atomic measure supported on $N$ ``data points'' in $H = {\rm R}^D$,
then $A_\mu$ is represented by an $N\times D$ matrix whose rows are
the data points, shifted by their center of mass. Therefore, in this case
$A_\mu$ maps data points in ${\reals}^D$ into $N$-dimensional feature
vectors (containing coefficients according to the dictionary of
shifted data points).
We remark that the dependence of $A_\mu$ on $\mu$ is not only due to the use of
$\xcm$, but also because the range of $A_\mu$ is in $L_2(\mu)$.

Next, we specify a kernel associated with $A_\mu$ and use it to conclude that $A_\mu$ is Hilbert-Schmidt.
Let us arbitrarily fix an orthonormal basis of $H$, $\{e_n\}_{n \in \nats}$, and express $A_\mu$ as follows:
\be
\label{eq:kernel_for_A}
(A_\mu y)(x) =\sum_{n \in \nats} \langle y,e_n \rangle \langle e_n,x-\xcm \rangle \ \mbox{for
all}\ x,y \in H .
\ee
We can thus view it as operator from $\ell_2$ (with the counting measure $\mu_{\sharp}$) to $L_2(\mu)$ with the kernel $k(x,n)= \langle e_n,x-\xcm \rangle$.
We note that this kernel is in $L_2(\mu_{\sharp} \times \mu)$, indeed, using the fact that the support of $\mu$ is bounded we obtain that
\be
\int \sum_{n \in \nats} |\langle e_n,x-\xcm \rangle |^2 \di \mu(x) = \int \|x-\xcm\|^2 \di \mu(x) < \infty.
\ee
We thus conclude that $A_\mu$ is Hilbert-Schmidt and in particular compact
(see e.g., \cite[Section~4]{halmos_book78}).

Since $A_\mu$ is compact, we can apply its SVD~\cite[Section~1.6.2]{yafaev_book92}.
We denote the singular values of $A_\mu$ repeated according
to multiplicities by $\{\sigma_i\}_{i \in \nats}$ .
Their  corresponding right vectors are denoted by $\{ v_i \}_{i \in \nats}$.
Equivalently, these are the orthonormal eigenvectors of $A^*_\mu A_\mu$ ($A^*_\mu$ is the adjoint
of $A_\mu$) with
eigenvalues $\{\sigma_i^2\}_{i \in \nats}$.
In Section~\ref{sec:pca} we apply the finiteness of $\sum_{i \in \nats} \sigma_i^2$,
which is equivalent to the Hilbert-Schmidt property of $A_\mu$.

\subsection{Least Squares $d$-Flats by SVD of the Data-to-Features Operator}\label{sec:pca}
We use the SVD of $A_\mu$ to construct a LS $d$-flat and express its corresponding error as
follows:

\begin{proposition}
\label{prop:pca}
A LS $d$-flat for $\mu$ exists and is obtained by
$$
\xcm + {\rm Sp}\{v_1,\ldots ,v_d\},
$$
where $v_1,\ldots ,v_d$ are the top right vectors of the
data-to-features operator $A_\mu$. It is unique if and only if
$\sigma_d>\sigma_{d+1}$. Moreover,
\be \label{eq:ls_error}
e_2(\mu,d) = \sqrt{\sum_{i > d} \sigma_i^2} \,. \ee
\end{proposition}

\begin{proof}
We express the function $F(L)$ of~\eqref{eq:def_FL} in terms of a
shift vector $c \in H$, a linear subspace $V \subseteq H$ and also
in terms of the orthogonal projection of $H$ onto the orthogonal
complement of $V$, which we denote by $P^\bot_V$. That is,
\begin{equation}
F(L) \equiv F(c,V) =  \int \mbox{dist}^2(x,c+V) \di \mu (x) =
\int\Vert P^\bot_V(x-c)\Vert^2\di \mu (x).
\end{equation}
We further note that
\be
F(c,V) =  \int\Vert P^\bot_V(x-\xcm)\Vert^2 \di \mu (x) + \Vert
P^\bot_V(c-\xcm)\Vert^2 \di \mu (x).
\ee
We thus conclude that the vector $c = \xcm$ minimizes $F(c,V)$
independently of $V$ (more generally, the set of minimizers is $\xcm
+ V$).

We next note that
\be
\min_{V} \int\Vert P^\bot_V(x-\xcm)\Vert^2\di \mu (x) =
\int\Vert x-\xcm \Vert^2\di \mu (x) - \max_V \int\Vert
P_V(x-\xcm)\Vert^2 \di \mu (x)\,,
\ee
where $P_V$ is the projection operator of $H$ onto $V$. Therefore,
instead of minimizing $F(\xcm,V)$, we maximize the function
\be
\label{eq:def_G}
G(V) = \int\Vert P_V(x-\xcm)\Vert^2 \di \mu (x) = \mbox{trace} (P_V
A^*_\mu A_\mu P^*_V) \,.
\ee
The last equality in~\eqref{eq:def_G} is evident due to the following expression of the adjoint operator $A_\mu^* : L_2(\mu) \to H$:
\be
\label{eq:adjoint}
A^*_\mu f = \int (x-\xcm)f(x)\di \mu (x) \ \mbox{for all} \ f\in
L_2(\mu).
\ee
Indeed, if $\{e_n\}_{n=1}^{\dim(V)}$ is an orthonormal basis of $V$ and $1 \leq n \leq \dim(V)$, then
\be
\langle e_n, P_V
A^*_\mu A_\mu P^*_V e_n \rangle = \int \langle e_n,x-\xcm \rangle^2 \di \mu(x).
\ee
Thus, summing both the LHS and RHS over $n=1, \ldots, \dim(V)$ we obtain the desired equality.

At last, we apply a theorem by Ky-Fan~\cite{ky_fan_eigenvalues} (see
also~\cite[Theorem~3.5]{trac_det_book}) to conclude that the maximum of $G$
is attained at $V := \mbox{Sp}\{v_1,\ldots ,v_d\}$, where
$v_1,\ldots ,v_d$ are the top eigenvectors of $A^*_\mu A_\mu$ and it
is unique if and only if $\sigma_d
> \sigma_{d+1}$. That is, $\xcm+\mbox{Sp}\{v_1,\ldots ,v_d\}$ is a
LS $d$-flat and unique whenever $\sigma_d
> \sigma_{d+1}$.
Furthermore,
$$
e_2^2(\mu,d) = \min_{c,V} F(c,V) =
\mbox{trace} (A^*_\mu A_\mu)
 - \max_V \mbox{trace} (P_V A^*_\mu A_\mu P^*_V) = \sum_{i>d} \sigma_i^2 \,.
$$

\end{proof}

\section{Examples of Geometric Condition Numbers on $H^{d+2}$} \label{sec:more-LS-curvatures}
%


In addition to the GCN
$\cvol$ defined in~\eqref{volume-1},
we suggest four other GCNs of $d+2$ variables.
Two of these squared GCNs are also scaled versions of this volume. The first one has the form
\begin{equation}\label{volume-2}\cvolk(X)=\frac{\MM_{d+1}(X)}{\diam^d(\mu)}.\end{equation}
The second one uses the $d$-dimensional
polar sine~\cite{LW-semimetric}. For $0\leq i\leq d+1$, the polar
sine  of   $X=(x_0,\ldots,x_{d+1})$ with respect to the
coordinate $x_i$  is
\begin{equation}\label{eq:def-psin}
\pds_{x_i}(X) = \begin{cases}\displaystyle\frac{\MM_{d+1}(X)}{\prod_{\substack{0\leq j\leq d+1\\
j\not=i}}\|x_j-x_i\|},& \textup{ if } \min(X)> 0;\\
 \qquad\qquad 0, & \textup{ otherwise.}
\end{cases}
\end{equation}
The corresponding polar GCN has the form:
\begin{equation}\cpols(X)=\diam(X)\,\sqrt{\frac{\sum_{i=0}^{d+1}\pds^2_{x_i}(X)}{d+2}}.\end{equation}

Another GCN is obtained by the $d$-dimensional LS error of the empirical measure associated with $X$
as follows:
\begin{equation}\label{eq:discrete-least-squares}
\cdlss(X)= \min_{L \in \agh}\sqrt{\frac{\sum_{i=0}^{d+2}\dist^2(x_i,L)}{d+2}}.\end{equation}
At last, we form the minimal height GCN:
\begin{equation}\label{min-height}\hmin(X)=\min_{0\leq i\leq
d+1}\dist(x_i,L[X(i)]).\end{equation}
We note that this GCN
is practically comparable to an $\ell_\infty$ version of the
$\ell_2$ GCN, $\cdlss$.
One can also form $\ell_p$ versions of such GCNs for all
$1 \leq p < \infty$, i.e., taking the $p$-th root of the average of
$p$-th powers of the distances.

The five GCNs on $H^{d+2}$ of this paper satisfy a variety of pointwise comparisons. For example,
via the product formula of~\eqref{product-formula}, as well
as~\cite[eqs.~(16), (118)]{LW-part1} for
arbitrary $X$ we have that
\begin{equation}\label{eq:control_vold_below1}\cvolk(X)\leq \cvol(X)\leq
\hmin(X)\leq\frac{(d+2)^\frac{3}{2}}{\sqrt{2}} \cdot
\cdlss(X).\end{equation}
Furthermore,  from the definitions
above we also have that
\be\cvolk(X)\leq\cvol(X)\leq\cpols(X).\label{eq:control_vold_below2}\ee

In order to control integrals of $\cpols$ by integrals of $\cdlss$,
we will use the following inequality of~\cite[Proposition $3.2$]{LW-part1}:
\begin{equation}\label{eq:opposite-control} \diam(X) \pds_{x_0}(X)
\leq \sqrt{2} \cdot (d+1) \cdot (d+2)^\frac{3}{2} \cdot
\frac{1}{\SCale_{x_0}(X)} \cdot \cdlss(X),\end{equation}
where $\SCale_{x_0}(X)$ was defined in~\eqref{eq:def_scale}.

\section{Upper Bounds on $e_2^2(\mu,d)$}
\label{sec:upper_bound}

\subsection{On $d$-Separated Measures} \label{sec:d_separated}
The $d$-separated measures form the weakest class of probability measures for which we can bound $e_2^2(\mu,d)$ by integrals of squared GCNs.
Let $\Supp$ denote the support of $\mu$ and $\diam(\mu)$ denote the diameter of this support.
We say that a $d$-simplex $X=(x_0,\ldots,x_{d}) \in
\supp(\mu)^{d+1}$  is $d$-separated (for $\omega>0$) if
\be
\MM_d(X) \geq \omega\cdot \diam(\mu)^d.
\label{eq:wds_simplex}
\ee
We say that the measure $\mu$ is $d$-separated  (with positive constants
$\omega$ and $\epsilon$) if there exist sets $ V_i \subseteq \supp(\mu)$, $0\leq i\leq d,$  that support $d$-separated $d$-simplices in the following way:
\begin{eqnarray}
 1. && \mu( V_i
) \geq \epsilon \text{ for each }
0\leq i\leq d.    \label{eq:d_sep1} \\
2. &&
\prod_{i=0}^{d}
 V_i\subseteq \left\{X\in \supp(\mu)^{d+1}: \MM_d(X)\geq\omega\cdot \diam(\mu)^d\right\}.
\label{eq:d_sep3}
\end{eqnarray}
The sets $V_i $ can be taken to be balls 
but this is not necessary and can be too
restrictive.

We also say that $\mu$ is $d$-separated with respect to the center of mass of $\mu$, $\xcm$,
or equivalently centrally $d$-separated,
if there exist sets
$V_i \subseteq \supp(\mu)$, $1\leq i\leq d$, satisfying~\eqref{eq:d_sep1} for $1\leq i\leq d$
as well as the following modification of~\eqref{eq:d_sep3}:
$$\prod_{i=1}^{d}V_i\times \{\xcm\}
 \subseteq \left\{X\in \supp(\mu)^{d+1}: \MM_d(X)\geq\omega\cdot \diam(\mu)^d\right\}.$$

The following lemma shows
that $d$-separation is a very general quantitative property in terms of
information about $e_2(\mu,d-1)$.
\begin{lemma}\label{lem:either-or} If $\mu$ is an arbitrary Borel
probability measure on $H$, then the following statements are equivalent:
\begin{enumerate}
\item $\mu$ is $d$-separated.
\item There exists a $d$-simplex $X \in \Supp^{d+1}$ such that $\MM_d(X)>0$.
\item $e_2(\mu,d-1)>0$.
\item $\mu$ is centrally $d$-separated.
\end{enumerate}
\end{lemma}

\begin{proof}
The equivalence of the first two statements of the lemma immediately follows from the continuity of $\MM_d$ and the following elementary observation
(where $B(x,r)$ is the closed  ball centered at $x$ of radius $r$):
$$\Supp=\{x\in H: \mu(B(x,r))>0\textup{ for all }r>0\}.$$

To establish the equivalence of the second and third statements we first note that  one direction is trivial.  That is, $e_2(\mu,d-1)=0$ implies that $\MM_d(X)=0$ for all
$X\in\Supp^{d+1}$, since all vertices will be trapped in a $(d-1)$-dimensional minimizing space $L$.

The other direction, i.e., $e_2(\mu,d-1)>0$ implies that $\MM_d(X)>0$ for some $X\in\Supp^{d+1}$, can be established by noting that for any affine space $L$ of dimension lesser than or equal to $d-1$ we have that
$$\int\dist^2(x,L)\di\mu(x)\geq e_2^2(\mu,d-1).$$
Using this observation we can construct $d$ points, $x_0,\ldots, x_{d-1}\in\Supp$ such that  for the $(d-1)$-simplex $X(d)=(x_0,\ldots,x_{d-1})$ we have that $\MM_{d-1}(X(d))>0$.
Since
$$\int\dist^2(x,L[X(d)])\di\mu(x)>0,$$
we can select another point $x_d\in L[X(d)]^c\cap\Supp$ (where $L[X(d)]^c$ is the complement of $L[X(d)]^c$) and taking $X=X(x_d,d)\in\Supp^{d+1}$ we conclude that $\MM_d(X)>0$.

The equivalence between the third and the fourth statements is proven in exactly the same way (recalling that $\xcm$ is contained in any LS $d$-flat).
\end{proof}

\subsection{The Main Theorem for Upper Bounds on $e_2^2(\mu,d)$}

Since all GCNs with $d+2$ variables suggested here control $\cvolk$ (see~\eqref{eq:control_vold_below1}), we only need to
bound $e_2^2(\mu,d)$ by an integral of $\cvolk^2$.
\begin{theorem}\label{thm:main-1} If $\mu$ is $d$-separated for the
positive constants $\omega$ and $\epsilon$, then
\be \label{eq:thm_main-1}
e_2^2(\mu,d) \leq \frac{1}{\omega^2\cdot\epsilon^{d+1}} \, \int_{H^{d+2}}\cvolker(X)\di\mu^{d+2}(X).\ee
\end{theorem}

\begin{proof}

We arbitrarily fix $\widetilde{X}(d+1)=(\widetilde{x}_0,\ldots,\widetilde{x}_d)\in\prod_{i=0}^d
\, V_i$, where $\{V_i\}_{i=0}^{d}$ are the sets of~\eqref{eq:d_sep1} and~\eqref{eq:d_sep3} defining the $d$-separated measure
$\mu$ (with constants $\omega$ and $\epsilon$).
It follows from both~\eqref{eq:d_sep3} and~\eqref{product-formula} that for $\tilde{X}(d+1)$ and $\tilde{X}(y,d+1)$ as in~\eqref{eq:def-removal1} and~\eqref{eq:def-removal2},
\begin{equation}\label{eq:matador}\cvolker(\widetilde{X}(y,d+1))=
\left(\frac{\MM_{d+1}(\widetilde{X}(y,d+1))}{\diam^d(\mu)}\right)^2\geq\\\omega^2\cdot\dist^2
(y,L[\widetilde{X}(d+1)])\,\end{equation}
and consequently
\begin{equation}\label{eq:simple_cheb1}
\frac{1}{\omega^2}\cdot\int_H\cvolker(\widetilde{X}(y,d+1))\di\mu(y)\geq\int_H \dist^2
(y,L[\widetilde{X}(d+1)])\di\mu(y)\geq e_2^2(\mu,d).
\end{equation}

For $0<\rho<\infty$ let
\begin{multline}\label{eq:simple_cheb2}\mathcal{E}_\rho=\Bigg\{\widetilde{X}(d+1)\in\prod_{i=0}^d\,
 V_i:\int_{H}\cvolker(\widetilde{X}(y,d+1))\di
\mu(y)\leq \rho \, \int_{H^{d+2}}\cvolker(X)\di\mu^{d+2}(X) \Bigg\}.\end{multline}
By Chebyshev's inequality we have $$\mu^{d+1}(\mathcal{E}_\rho)\geq \mu^{d+1}(\prod_{i=o}^d V_i)-\frac{1}{\rho}.$$
Then, taking $\rho>\frac{1}{\epsilon^{d+1}}$
forces $\mathcal{E}_\rho\not=\emptyset$.  Thus, restricting $\widetilde{X}(d+1)$ to $\mathcal{E}_\rho$ for $\rho>\frac{1}{\epsilon^{d+1}}$, and combining equations~\eqref{eq:simple_cheb1} and~\eqref{eq:simple_cheb2}, we conclude that the inequality of Theorem~\ref{thm:main-1} holds with the controlling constant ${\rho}/{\omega^2}$.
Since this holds for arbitrary such $\rho$ we obtain the constant given in Theorem~\ref{thm:main-1}.
\end{proof}

\section{Lower Bounds for $e_2^2(\mu,d)$}
\label{sec:lower_bound}

We first verify a lower bound on
$e_2^2(\mu,d)$ by an integral of $\cdlss^2$. Since the GCNs $\cvolk$, $\cvol$ and $\hmin$ are controlled by $\cdlss$
(see~\eqref{eq:control_vold_below2}), this bound also holds for all of these GCNs.
\begin{proposition}\label{prop:main-2}If $\mu$ is an arbitrary Borel probability measure on $H$, then
\be\int_{H^{d+2}}\cdls2(X)\,\di\mu^{d+2}(X)
\leq e^2_2(\mu,d).\label{eq:prop_main-2}\ee\end{proposition}

\begin{proof}
For any fixed $d$-flat $L\in \agh$, by the definition of the GCN
$\cdlss(X)$ and a subsequent application of Fubini's Theorem we obtain that
\begin{equation*}\int_{H^{d+2}}\cdls2(X)\di\mu^{d+2}(X)\leq\frac{1}{d+2}\sum_{i=0}^{d+1}
\int_{H^{d+2}}\dist^2((X)_i,L)\di\mu^{d+2}(X) = \int_{H}\dist^2(x,L)\di\mu(x).\end{equation*}
The proposition is concluded by taking the infimum over all $L\in \agh$.
\end{proof}

A lower bound on $e_2^2(\mu,d)$ in terms of an integral of the GCN
$\cpols^2$ requires the following notions of regularity of $\mu$.
For $\gamma>0$, we say that $\mu$ is $\gamma$-regular if there exists a
 $C\geq1$ such that
\begin{equation}\label{gamma-regular}
\frac{t^\gamma}{C} \leq \mu(B(x,t)) \leq C \cdot t^\gamma \ \textup{
for all }x\in\Supp,\ 0<t\leq\diam(\mu).\end{equation}
We say that $\mu$ is $\gamma$-upper-regular if the upper bound
of~\eqref{gamma-regular} holds. We
call the minimal such constant $C$ satisfying~\eqref{gamma-regular}
(or its right hand side for upper-regular measures) the {\em
regularity constant} of $\mu$.

Using these notions we formulate the following lower bound on $e_2(\mu,d)$
and verify it in the following section.
\begin{theorem}\label{thm:psin_curv_bound}
If $\mu$ satisfies either one of the following conditions: $\mu$ is $\gamma$-upper-regular for $\gamma>2$ or
$\mu$ is $\gamma$-regular for $\gamma > 1$ with $d=1$,
then
\be
\label{eq:psin_upper_bound1}\int_{H^{d+2}}\cpol2(X)\di\mu^{d+2}(X)\lessapprox
e^2_2(\mu,d),\ee
where the comparison only depends on $d$, $\diam(\mu)$ and the regularity constant of
$\mu$.
\end{theorem}

\subsection{Proof of Theorem~\ref{thm:psin_curv_bound}}

The proof of this proposition is technically detailed, however it is based on few elementary ideas.
It starts by replacing the integral of
$\cpols^2(X)$ with the integral of
 $\diam^2(X) \cdot \pds^2_{x_0}(X)$ by using a change of variables. Next, in
view of~\eqref{eq:opposite-control}, $\diam(X) \cdot \pds_{x_0}(X)$ is
controlled by $\cdlss(X)$ for well-scaled simplices at $x_0$ (recalling that these are simplices for which the minimal edge length at $x_0$
is comparable to the maximal edge length at $x_0$; see Section~\ref{sec:context}).
Therefore, by applying  Proposition~\ref{prop:main-2}, the
integral of $\diam^2(X) \cdot \pds^2_{x_0}(X)$ over well-scaled
simplices is controlled by $e^2_2(\mu,d)$.

The proof thus only requires
the control of the integral of $\diam^2(X) \cdot \pds^2_{x_0}(X)$ on poorly scaled simplices in $H^{d+2}$ (i.e., simplices which are not well-scaled).
The idea  follows
the procedure of geometric
multipoles~\cite[Section~$9$]{LW-part1}, which uses a
multiscale decomposition of the integral and finds local control according to the
goodness of approximation by $d$-flats at different scales and
locations. While in~\cite{LW-part1} we sought
local control in terms of multiscale best fit $d$-flats, in the
current work we seek local control in terms of a global best fit
$d$-flat.

\subsubsection{Preliminary Notation and Conventions}
\label{subsec:notation}
For simplicity, we assume throughout the proof that $\diam(\mu)=1$ and thus suppress estimates depending on $\diam(\mu)$.

For $X \in H^{d+2}$ with $x_0 = (X)_0$, we frequently refer to $\min(X)$, $\Min0(X)$, $\Max0(X)$ and $\SCale_{x_0}(X)$ defined
in Section~\ref{sec:context}.
We decompose the set of simplices with non-zero edge lengths
according to the following sets indexed by $k,i\in\mathbb{N}_0$:
\begin{multline}\label{eq:decomp-1}S_{i,k}=\Big\{X=(x_0,\ldots,x_{d+1})\in H^{d+2}:\Max0(X)\in
(1/2^{i+1},1/2^i] \\\textup{ and
}\SCale_{x_0}(X)\in(1/2^{k+1},1/2^k]\Big\}.\end{multline}
We will mainly use their following subsets:
\begin{equation} \nonumber \label{decomp-2}S'_{i,k}=\left\{X\in
S_{i,k}:\Min0(X)=\|x_1-x_0\|\textup{ and
}\Max0(X)=\|x_2-x_0\|\right\}.\end{equation}

For $x_0\in H$ and $\ell\in\mathbb{N}_0$ we denote the annulus
centered at $x_0$ and of ``radius'' $1/2^\ell$ by
\begin{equation}A(x_0,\ell)=\{x\in H:1/2^{\ell+2}<\|x-x_0\|\leq1/2^\ell\},\end{equation}
and we note that for all $X\in S_{i,k}$ and fixed $1\leq j\leq d+1$
we have that
\begin{equation}\label{eq:inclusion?}(X)_j\in
A(x_0,\ell)\textup{ for some }i\leq\ell\leq i+k\textup{ depending on
}X\textup{ and }j.\end{equation}

\subsubsection{Case I: $\mu$ Upper-Regular for $\gamma>2$} \label{sec:case1}
We decompose the integral of $\cpols$ by using the sets of~\eqref{eq:decomp-1} and applying symmetry properties of the polar sine:
\begin{multline}\label{eq:initial-decomp}\int_{H^{d+2}}\cpol2(X)\di\mu^{d+2}(X)=\int_{H^{d+2}}\diam^2(X)\pds^2_{x_0}(X)\di\mu^{d+2}(X)=\\
\sum_{i=0}^\infty\sum_{k=0}^\infty\int_{S_{i,k}}\diam^2(X)\pds^2_{x_0}(X)\di\mu^{d+2}(X)=\\
d\cdot (d+1) \, \sum_{i=0}^\infty\sum_{k=0}^\infty \int_{S'_{i,k}}\diam^2(X)\pds_{x_0}^2(X)\di\mu^{d+2}(X).
\end{multline}
The elements of the last double sum of~\eqref{eq:initial-decomp} that correspond to
$k=0$ can be controlled by combining~\eqref{eq:opposite-control} (where
here $\SCale_{x_0}(X) \geq 1/2$) and Proposition~\ref{prop:main-2},
thus obtaining
\begin{multline}\label{corpse_whitey}\sum_{i=0}^\infty\int_{S'_{i,0}} \diam^2(X) \pds^2_{x_0}(X)\di\mu^{d+2}(X)=
\int_{\cup_{i=0}^\infty
S'_{i,0}}\diam^2(X)\pds^2_{x_0}(X)\di\mu^{d+2}(X)\\
\lessapprox e_2^2(\mu,d).\end{multline}
%

We now find sufficient bounds for the other
terms in the last double sum~\eqref{eq:initial-decomp} to obtain
convergence in $i$ and $k$.
Applying~\eqref{eq:opposite-control} to a fixed term on the last double sum
of~\eqref{eq:initial-decomp} we obtain the following bound for an
arbitrary $d$-flat $L$:
\begin{multline}\label{fixed-1-2}
\int_{S'_{i,k}}\diam^2(X)\cdot\pds^2_{x_0}(X)\di\mu^{d+2}(X)\lessapprox\\\sum_{j=0}^{d+1}\int_{S'_{i,k}}
\frac{\dist^2(x_j,L)}{\SCale^2_{x_0}(X)}\di\mu^{d+2}(X)\leq
2^{2k}\cdot\sum_{j=0}^{d+1}\int_{S'_{i,k}}
\dist^2(x_j,L)\di\mu^{d+2}(X).\end{multline}

We claim that for all $0 \leq j \leq d+1$:
\begin{equation}\label{eq:all_j}\inf_{L \in \agh}\int_{S'_{i,k}}
 \dist^2(x_j,L)\di\mu^{d+2}(X)\lessapprox(1/2^{k+i})^\gamma(1/2^i)^{\gamma\cdot
d}\cdot e^2_2(\mu,d).\end{equation}
To see this it is sufficient to integrate with respect to $x_j$ last,
and depending on the index $j$,  to vary the order of integration of the other variables slightly.
 If $j>0$ then we take the integration with respect to $x_0$ as the second to last integration.
If $j>1$, then we  integrate with respect to
$x_1$, then $x_0$ and then finally $x_j$.   Following this procedure
\eqref{eq:all_j} clearly follows
from the combination of~\eqref{eq:inclusion?} with the
upper-regularity. Indeed, the factor $(1/2^{k+i})^\gamma$ arises
from the integration over the coordinate $x_1$ if $j \neq 1$ and
$x_0$ if $j=1$, $(1/2^i)^{\gamma\cdot d}$ from the rest of
coordinates excluding $x_j$ and clearly $e^2_2(\mu,d)$ from
the coordinate $x_j$.

Applying~\eqref{eq:all_j} to the RHS of~\eqref{fixed-1-2}, we obtain that
\begin{equation}\label{good-boy}
\int_{S'_{i,k}}\diam^2(X)\cdot\pds^2_{x_0}(X)\di\mu^{d+2}(X)
\lessapprox
1/2^{(\gamma-2)\cdot k}
\cdot1/2^{\gamma\cdot (d+1) \cdot i}\cdot
e^2_2(\mu,d).\end{equation}

Finally, combining~\eqref{eq:initial-decomp}, \eqref{corpse_whitey}
and~\eqref{good-boy} we conclude that
\begin{multline*}\int_{H^{d+2}}\diam^2(X)\pds^2_{x_0}(X)\di\mu^{d+2}(X)=
\sum_{i=0}^\infty\sum_{k=0}^\infty\int_{S_{i,k}}\diam^2(X)\pds^2_{x_0}(X)\di\mu^{d+2}(X)
\lessapprox\\
\left(
1 +
\sum_{i=0}^\infty\sum_{k=1}^\infty1/2^{(\gamma-2)\cdot k}
\cdot1/2^{\gamma\cdot (d+1)\cdot i}\right)\cdot
e^2_2(\mu,d).\end{multline*}
Since the coefficient on the RHS above
is clearly finite for $\gamma>2$, the proposition is thus proved for the current case.

\subsubsection{Case II: $\mu$ is $\gamma$-Regular for $\gamma>1$ and $d=1$}%
\label{sec:case2}
Since $d=1$ we work with triangles $X=(x_0,x_1,x_2)\in H^3$ and we need to prove that
\be\int_{H^3}\diam^2(X) \sin^2_{x_0}(X) \di\mu^3(X)\lessapprox
e^2_2(\mu,1).\label{eq:psin_upper_bound2}\ee
The procedure here is similar to that
of  Section~\ref{sec:case1}, however we must use  an inequality for
the sine function that holds  with high  probability for a $\gamma$-regular $\mu$ with  $\gamma>1$.
We clarify this  as follows.

For fixed $X=(x_0,x_1,x_2)\in H^3$, $X(u,1)$ as in~\eqref{eq:def-removal2}, and $C\geq1$,  let
\begin{equation}\label{one-term-set}U(X,C):=\left\{u\in H: |\sin_{x_0}(X)|\leq C\cdot |\sin_{x_0}(X(u,1))|\right\},\end{equation}
and  for $\alpha>0$ let $A_{\alpha}(X,C)$ denote the
restriction of $U(X,C)$ to an annulus:
\be
A_{\alpha}(X,C):=U(X,C)\cap
B(x_0,\Max0(X))\setminus B(x_0,\alpha\cdot\Max0(X)).
\label{eq:def_A}\ee
The following lemma shows that the defining inequality of~\eqref{one-term-set} occurs with high probability (we delay its proof to Section~\ref{app:one-term}).
\begin{lemma}\label{lemma:local-inequality}If
$\mu$ is $\gamma$-regular for $\gamma>1$ with regularity constant $C_\mu$, and the constants $C_0$ and $\alpha_0$ are such that
\be \label{eq:alpha_bound}
C_0 \geq
\half \cdot \left(4\cdot 5^{\gamma/2} \cdot C_\mu^2\right)^{\frac{1}{\gamma-1}}
\ \textup{ and } \
0<\alpha_0\leq\left({4\cdot
C^2_{\mu}}\right)^{-1/\gamma},\ee
then the following inequality holds
uniformly for all $X\in\supp(\mu)^3$:
$$\mu\left(A_{\alpha_0}(X,C_0)\right)\geq\frac{1}{2}
\cdot\mu\left(B(x_0,\Max0(X)\right).$$\end{lemma}
For the rest of this section
we use the optimal values of the constants $C_0$ and $\alpha_0$ in~\eqref{eq:alpha_bound} (i.e., the lower bound
for $C_0$ and upper bound for $\alpha_0$).
We decompose all triangles with non-zero edge lengths into
the sets
$$S_k=\{X\in H^3:\SCale_{x_0}(X)\in(\alpha_0^{k+1},\alpha_0^k]\},$$
for  $ k \geq 1,$ and we denote
$$S_k' =\{X\in S_k:\Max0(X)=\|x_2-x_0\|\}\subset S_k.$$

By the symmetry of $|\sin_{x_0}(X)|$ with respect to $x_1$ and $x_2$, we note that
\begin{multline}\label{d=1-decompose}\int_{H^3}\diam^2(X)\cdot\sin^2_{x_0}(X)
\di\mu^3(X)
=\sum_{k=0}^{\infty}\int_{S_k}\diam^2(X)\cdot\sin^2_{x_0}(X)
\di\mu^3(X)
\leq\\
2 \, \sum_{k=0}^{\infty}\int_{S'_k}\diam^2(X)\cdot\sin^2_{x_0}(X)
\di\mu^3(X).
\end{multline}
We note that $X\in S'_0$ is well-scaled (for $\alpha_0$), and by combining~\eqref{eq:opposite-control}  and Proposition~\ref{prop:main-2} we see that
$$\int_{S'_0}\diam^2(X)\cdot\sin^2_{x_0}(X)
\di\mu^3(X)\lessapprox  e_2^2(\mu,1). $$
We now use Lemma~\ref{lemma:local-inequality} to control the individual terms for $k\geq 1$ on the RHS of~\eqref{d=1-decompose}.
We arbitrarily fix $X\in S_k'$ and define the probability measure
$$\tilde{\mu}_X:=\frac{\mu|_{A_{\alpha_0}(X,C_0)}}{\mu\left(A_{\alpha_0}(X,C_0)\right)}.$$
We note that for any $y \in
A_{\alpha_0}(X,C_0)$:
\begin{equation}\nonumber\diam^2(X)\cdot\sin^2_{x_0}(X)\leq C_0^2 \cdot \diam^2(X(y,1))\cdot\sin^2_{x_0}(X(y,1)).
\end{equation}
and consequently for $X(y,1)$ as in~\eqref{eq:def-removal2},
\begin{equation}\label{prob-ineq}\diam^2(X)\cdot\sin^2_{x_0}(X)\leq C_0^2\int_{A_{\alpha_0}(X,C_0)}\diam^2(X(y,1))\cdot\sin^2_{x_0}(X(y,1))\di\tilde{\mu}_X(y).
\end{equation}
Since the triangle $X(y,1)$ is well-scaled for each $y\in A_{\alpha_0}(X,C_0)$, we can apply the inequality of~\eqref{eq:opposite-control} to the integrand on the
 RHS of~\eqref{prob-ineq} to obtain
 \begin{equation}\label{prill}\diam^2(X)\cdot\sin^2_{x_0}(X)\lessapprox  \dist^2(x_0,L)+ \int_{A_{\alpha_0}(X,C_0)}\dist^2(y,L)\di\tilde{\mu}_X(y) +\dist^2(x_2,L)
 \end{equation}
for any $L \in \mathrm{AG}_{1}(H)$, where the constant of the inequality is independent of $k$.

Fixing the line $L$,   the middle term on the RHS of~\eqref{prill} trivially has the bound
\begin{equation}\label{needed-again}\int_{A_{\alpha_0}(X,C_0)}\dist^2(y,L)\di\tilde{\mu}_X(y)\leq \frac{1}{\mu(A_{\alpha_0}(X,C_0))}\int_{H}\dist^2(x,L)\di\mu(x).\end{equation}
Thus, applying~\eqref{prill} to~\eqref{d=1-decompose}, and then~\eqref{needed-again} to~\eqref{prill},   for an arbitrary line $L$  we have
\begin{multline}\label{well-control}\int_{S_k'}\diam^2(X)\cdot\sin^2_{x_0}(X)
\di\mu^3(X)\lessapprox\\
\int_{S_k'}\dist^2(x_0,L)
\di\mu^3(X)+  \int_{H}\dist^2(x,L)\di\mu(x)\cdot \int_{S_k'} \frac{\di\mu^3(X)}{\mu(A_{\alpha_0}(X,C_0))}  +\int_{S_k'}\dist^2(x_2,L)
\di\mu^3(X).\end{multline}

We  bound the  terms  of~\eqref{well-control}
separately. The first term satisfies
\begin{multline}\label{eq:just_ineq1}\int_{ S_k' }\dist^2(x_0,L)\di\mu^3(X)
 \leq\\\int_{H^2}\dist^2(x_0,L)
\left(\int_{B(x_0,\alpha_0^k\cdot\|x_2-x_0\|)}\di\mu(x_1)\right)
\di\mu(x_2)\di\mu(x_0)\lessapprox\\ \alpha_0^{k\cdot\gamma}\cdot\int_H\dist^2(x,L)\di\mu(x).\end{multline}
A similar  computation gives the same bound for the third term.

Then, by Lemma~\ref{lemma:local-inequality} and the regularity of $\mu$  we have that $\mu\left(A_{\alpha_0}(X,C_0)\right)\gtrapprox\Max0(X)^\gamma$, and thus the second term of~\eqref{well-control} satisfies
\begin{multline}\label{eq:just_ineq23}\int_H\dist^2(x,L)\di\mu(x)\cdot\int_{S_k'} \frac{\di\mu^3(X)}{\mu(A_{\alpha_0}(X,C_0))}
 \lessapprox\\
\int_H\dist^2(x,L)\di\mu(x)\cdot
 \left(
\int_{S_k'}\frac{\di\mu^3(X)}{\|x_2-x_0\|^\gamma}\right)\lessapprox\\\alpha_0^{k\cdot\gamma}\cdot
\int_H\dist^2(x,L)\di\mu(x).\end{multline}
Applying~\eqref{eq:just_ineq1} and~\eqref{eq:just_ineq23}
 to the terms of~\eqref{well-control} we
have the bound
\begin{equation}\label{frantic-loop}\int_{S_k'}\diam^2(X)\cdot\sin^2_{x_0}(X)
\di\mu^3(X)\lessapprox\alpha_0^{k\cdot\gamma}\cdot
\int_H\dist^2(x,L)\di\mu(x).\end{equation}
Taking an infimum over $L\in \mathrm{AG}_{1}(H)$ on the RHS of~\eqref{frantic-loop}, and then
summing this inequality over $k\geq1$
we see that the proposition holds.

\subsubsection{Proof of Lemma~\ref{lemma:local-inequality}}\label{app:one-term}
Equation~\eqref{eq:alpha_bound} is a direct consequence of the following two equations:
\be\label{eq:oneterm1} \mu(B(x_0,\Max0(X)) \setminus
B(x_0,\alpha_0\cdot\Max0(X)))
\geq\frac{3}{4}\cdot\mu(B(x_0,\Max0(X))\ee
and
\be
\label{eq:oneterm2}
\mu(U(X,C_0))\geq\frac{3}{4}\cdot\mu(B(x_0,\Max0(X)).
\ee

The inequality of~\eqref{eq:oneterm1} follows from the $\gamma$-regularity of $\mu$ and the  constant $\alpha_0$.
Indeed,
\begin{multline*}\label{eq:oneterm1} \mu(B(x_0,\Max0(X)) \setminus
B(x_0,\alpha_0\cdot\Max0(X)))
\geq \\\mu(B(x_0,\Max0(X))) - C_{\mu} \cdot
\alpha_0^\gamma \cdot \Max0(X)^\gamma
\geq\frac{3}{4} \cdot\mu(B(x_0,\Max0(X)).\end{multline*}

We conclude by proving~\eqref{eq:oneterm2}.
We form the tube of radius $\Max0(X)/C_0$ on the line $L[X(1)]$,
$$\Tbe\left(L[X(1)],\Max0(X)/C_0\right)
=\left\{y:\dist(y,L[X(1)])\leq \Max0(X) / C_0 \right\},$$
and note that
\be\label{eq:inclusion_tube}(\Tbe\left(L[X(1)],\Max0(X)/C_0\right))^c\cap
B(x_0,\Max0(X))\subseteq U(X,C_0).\ee
Indeed, since $|\sin_{x_0}(X(u,1))| \cdot \|u-x_0\|=\dist(u,L[X(1)])$ we have the following lower bound for any
$u\in B(x_0,\Max0(X))$:
\begin{equation}\label{eq:sine_ineq_simple}|\sin_{x_0}(X(u,1))|\cdot \Max0(X) \geq\dist(u,L[X(1)]).\end{equation}
Applying~\eqref{eq:sine_ineq_simple} to $u\in(\Tbe\left(L[X(1)],\Max0(X)/C_0\right))^c\cap
B(x_0,\Max0(X))$, we obtain that
$$C_0 \cdot |\sin_{x_0}(X(u,1))| \geq 1\geq|\sin_{x_0}(X)|,$$
i.e., $u \in U(X,C_0)$ and~\eqref{eq:inclusion_tube} is concluded.

At last, we show that
\be \label{eq:lemma_last_measure}\mu\left(\Tbe\left(L[X(1)],\Max0(X)/C_0\right)\cap
B(x_0,\Max0(X))\right) \leq  \frac{1}{4} \cdot\mu(B(x_0,\Max0(X))\ee
and combining it with~\eqref{eq:inclusion_tube} we establish~\eqref{eq:oneterm2}.
We first note that the intersection of the tube $\Tbe (L[X(1)],\Max0(X)/C_0)$ with
$B(x_0,\Max0(X))$ can be covered by at most $2 \cdot C_0$ balls
of radius $\sqrt{5} \cdot\Max0(X)/ (2 \cdot C_0)$ and thus
\begin{multline} \label{eq:lemma_real_last}\mu\left(\Tbe\left(L[X(1)],\Max0(X)/C_0\right)\cap
B(x_0,\Max0(X))\right) \leq \\  C_\mu^2\cdot
{5}^{\gamma/2}\cdot (2 \cdot C_0)^{1-\gamma}\cdot\mu(B(x_0,\Max0(X)).\end{multline}
Equation~\eqref{eq:lemma_last_measure} and consequently~\eqref{eq:oneterm2} follows by combining~\eqref{eq:alpha_bound}
with~\eqref{eq:lemma_real_last}.

\section{GCNs on $H^{d+1}$ and $H^{d}$ and their Corresponding Comparisons}
\label{sec:gcn_reduce}
\subsection{$d$-Dimensional GCNs of Only $d+1$ Variables}
\label{subsec:cm}

If one knows a point that lies on a LS $d$-flat, then  any of the above GCNs can be reduced to a function of
only $d+1$ variables by arbitrarily fixing one of the original variables at that point.
We exemplify this idea with the center of mass, $\xcm$, which lies on the LS $d$-flat (see Proposition~\ref{prop:pca}) and later explain how to extend it to
other points.

We consider the set of $(d+1)$-simplices with a fixed vertex at $\xcm$, that is, we define the set  $H^{d+1}_\xcm=\{\xcm\}\times H^{d+1}$, and we restrict our attention to the elements
\begin{equation}\label{simple-note}X=(\xcm,x_1,\ldots,x_{d+1})\in  H^{d+1}_\xcm.\end{equation}
As such, we can replace any GCN, $c: H^{d+2} \rightarrow \reals,$ by $c(X): H^{d+1}_\xcm \rightarrow \reals$ and establish the relevant comparisons
as follows, where $\mu^{d+1}$ on $H^{d+1}_\xcm$ is clearly  on the set $H^{d+1}$.

\begin{proposition} \label{prop:bound-center-of-mass}
If $\mu$ is centrally $d$-separated (for $\omega$ and $\eps$)
then
\be \label{eq:lower-bound-center-of-mass} e_2^2(\mu,d)\leq
\frac{1}{\omega^2\cdot\epsilon^{d}}
\int_{H^{d+1}_\xcm}\cvolker(X)\di\mu^{d+1}(X).\ee

If on the other hand $\mu$ satisfies either one of the following conditions: $\mu$ is $\gamma$-upper-regular for $\gamma>2$ or
$\mu$ is $\gamma$-regular for $\gamma > 1$ with $d=1$,
then
\be
\label{eq:psin_upper_bound1_2}\int_{H^{d+1}_\xcm}\cpol2(X)\di\mu^{d+1}(X)\lessapprox
e^2_2(\mu,d),\ee
where the comparison only depends on $d$, $\diam(\mu)$ and the regularity constant of
$\mu$.

Moreover, if $\mu$ is an arbitrary Borel probability measure on $H$, then
\be \label{eq:upper-bound-center-of-mass}
\int_{H^{d+1}_\xcm}\cdls2(X)\di\mu^{d+1}(X)\leq e^2_2(\mu,d). \ee
\end{proposition}

\begin{proof}
The proofs of~\eqref{eq:lower-bound-center-of-mass} and~\eqref{eq:psin_upper_bound1_2} are identical to those of~\eqref{eq:thm_main-1} and~\eqref{eq:psin_upper_bound1} respectively, while they also use the fact
that $\xcm$ lies in the LS $d$-flat.

In order to prove~\eqref{eq:upper-bound-center-of-mass}, we apply~\eqref{eq:control_vold_below1} and obtain
the following for
any fixed  $L\in\agh$:
\begin{multline*}
\int_{H^{d+1}_\xcm}\cdls2(X)\di\mu^{d+1}(X )\leq\\
\frac{1}{d+2}\left(\sum_{i=1}^{d+1}\int_{H^{d+1}_\xcm}\dist^2(x_i,L)\di\mu^{d+1}(X)
+\int_{H^{d+1}_\xcm}\dist^2(\xcm,L)\di\mu^{d+1}(X)\right).\end{multline*}

Since the function $\dist^2(\cdot,L)$ is convex for fixed
$L$, we apply Jensen's Inequality to the last term on the RHS and
then Fubini's Theorem to all terms and thus conclude~\eqref{eq:upper-bound-center-of-mass}.
\end{proof}

We note that when using a fixed point $y$ on the LS $d$-flat instead of $\xcm$, then
few modifications are needed. First of all, the minimizations defining both $e_2(\mu,d)$ and $\cdlss$  need to be restricted to subspaces in $\agh$ containing the point $y$.
Also, $d$-separation needs to be defined w.r.t.~$y$ (instead of $\xcm$).
When clustering $d$-dimensional linear subspaces, the LSCC algorithm (linear SCC)~\cite{spectral_applied,
spectral_theory} applies such a strategy with $y=0$, which obviously lies on all linear subspaces.

\subsection{A $d$-Dimensional GCN of Only $d$ Variables}
\label{subsec:deshpande} The work of Deshpande et al.~\cite{DRVW06,
DV06} suggests a GCN on $H^d$, which we denote by $\cdes$.  The idea is to look at the geometry of $d$-simplices having the center of mass,  $\xcm$, as a fixed vertex.  As such, we make the definition $H^d_{\xcm}=\{\xcm\}\times H^d$, and we consider $d$-simplices
$$\tilde{X}=(\xcm,x_1,\ldots,x_d)\in H^d_\xcm.$$
We note that  $\tilde{X}\in  H^d_\xcm$ is simply the projection of some $X\in  H^{d+1}_\xcm$, i.e.~$\tilde{X}= X(d+1)$.

We define the square of
$\cdes(\tilde{X})$ in the following way:
$$
\cdesh2(\tilde{X})= \frac{\displaystyle\MM^2_{d}(\tilde{X}) \, \int_H
\dist^2(y,L[\tilde{X}])\di
\mu(y)}{\displaystyle\int_{ H^d_\xcm}\MM^2_{d}(\tilde{Y})\di \mu^{d} (\tilde{Y})},
$$
where $\mu^d$ on $H^d_\xcm$ is clearly taken on the set $H^d$.
For a fixed $\tilde{X}\in H^d_\xcm$, the  GCN $\cdesh2(\tilde{X})$ is simply the average squared volume of $(d+1)$-simplices having  $\tilde{X}$ as a $d$-dimensional face, divided by the average squared volume of $d$-simplices with the center of mass as a vertex.  This follows directly from~\eqref{product-formula}.

The comparison of $e_2^2(\mu,d)$ and the integral of $\cdesh2$ is established as follows.
\begin{theorem}
If $\mu$ is centrally $d$-separated with compact support, then
\be \label{eq:deshpande_more} e_2^2(\mu,d) \leq \frac{1}{\omega^2\cdot\epsilon^{d}}
 \int_{H^d_{\xcm}}\cdesh2(\tilde{X}) \di \mu^{d} (\tilde{X}).\ee
If on the other hand $\mu$ is an arbitrary Borel probability measure on $H$, then
\be \label{eq:deshpande} \int_{H^d_{\xcm}}\cdesh2(\tilde{X}) \di \mu^{d} (\tilde{X}) \leq e_2^2(\mu,d). \ee
\label{thm:deshpande}\end{theorem}

\begin{proof}
In order to simplify the argument, we introduce a related GCN on $(d+1)$-simplices with a vertex fixed at $\xcm$.
That is we look at $(d+1)$-simplices per~\eqref{simple-note}, i.e.,~$X\in H^{d+1}_\xcm,$ and the GCN $\cDn(X)$, whose square is defined by
\be \label{eq:def_cdnn}\cDnn2(X)=
\frac{\MM^2_{d+1}(X)}{\displaystyle\int_{ H^d_\xcm}\MM^2_{d}(\tilde{Y})\di
\mu^{d} (\tilde{Y})}.
\ee

Per~\eqref{product-formula} and  Fubini's Theorem we see that the corresponding integrals of the
two squared GCNs, $\cdesh2(\tilde{X})$ and $\cDnn2(X)$, are equal, i.e.,
\be
\label{eq:fubini_cDnn}
\int_{H^d_\xcm}\cdesh2(\tilde{X}) \di \mu^{d} (\tilde{X}) = \int_{H^{d+1}_\xcm}
\cDnn2(X) \di \mu^{d+1} (X).\ee
We will thus prove Theorem~\ref{thm:deshpande} with the simpler GCN $\cDn(X)$. Theorem~\ref{thm:sin_val_compare} will immediately follow from our estimates below.

We first note that~\eqref{eq:deshpande_more} follows from~\eqref{eq:lower-bound-center-of-mass} and \eqref{eq:fubini_cDnn},
as well as the following fact:
$$
\cvolk(X) \leq \cDn(X) \ \ \forall X \in H^{d+1}_\xcm.
$$
In order to prove~\eqref{eq:deshpande} we generalize~\cite[Lemma~3.1]{DRVW06} to our continuous setting by proving the identity:
\be \label{eq:mult_singulars}\int_{H^{d+1}_\xcm} \MM^2_{d+1}(X) \di
\mu^{d+1} (X)
= \sum_{1 \leq t_1 < \ldots < t_{d+1}} \sigma^2_{t_1} \cdots
\sigma^2_{t_{d+1}}, \ee
where $\{\sigma_i\}_{i \in \nats}$ are the singular values of the data-to-features operator $A_\mu$.
We obtain~\eqref{eq:mult_singulars} by expanding the expression $\det(I+\lambda A_\mu A^*_\mu)$ in $\lambda\in \reals$ in two different ways and equating the corresponding coefficients.

We first apply~\cite[Theorem~IV.6.1]{trac_det_book} to obtain that
\be
\label{eq:det_goh1}
\det(I+\lambda A_\mu A^*_\mu)=\prod_{j \in \nats}(1+\lambda \sigma_j^2)=1+\sum_{k \in \nats}
\sum_{j_1<\ldots<j_k \in \nats} \sigma_{j_1}^2 \cdots \sigma_{j_k}^2 \cdot \lambda^k.
\ee
Next, in view of~\eqref{eq:adjoint} we express the operator $A_\mu A^*_\mu: L_2(\mu) \to L_2(\mu)$ as follows:
\be
\label{eq:adjoint2}
(A_\mu A^*_\mu f)(y) = \int \langle x-\xcm,\ y-\xcm \rangle f(x)\di \mu (x), \
\mbox{for all } f\in L_2(\mu) \text{ and } y \in H \,
\ee
so that it has the kernel
\be
k(y,x) = \langle x-\xcm,\ y-\xcm \rangle.
\ee
By adapting the proof of~\cite[Theorem~VI.1.1]{trac_det_book}
to the operator $\lambda A_\mu A^*_\mu$ with the kernel
$\lambda k(y,x)$ and the compactly supported measure $\mu$ we obtain that
\begin{multline}
\label{eq:det_goh2}
\det(I+\lambda A_\mu A^*_\mu)=1+\sum_{m\in \nats} \frac{\lambda^m}{m!} \int_{H^m}
\det (\{k(x_i,x_j)\}_{i,j=1}^{m}) \di\mu(x_1) \ldots \di\mu(x_m)=\\
            1+\sum_{m\in \nats} \lambda^m \int_{H^m_\xcm} \MM^2_{m}(Y) \di
\mu^{m} (Y),
\end{multline}
where $H^m_\xcm=\{\xcm\}\times H^m$.
Equation~\eqref{eq:mult_singulars} thus immediately follows from both~\eqref{eq:det_goh1} and~\eqref{eq:det_goh2}.

We will also use the following immediate estimate:
\be \label{eq:desh1}\sum_{1 \leq t_1 < \ldots < t_{d+1}}
\sigma^2_{t_1} \cdots \sigma^2_{t_{d+1}} \leq \sum_{1 \leq t_1 <
\ldots < t_{d}} \sigma^2_{t_1} \cdots \sigma^2_{t_{d}}
\sum_{j=d+1}^\infty \sigma_j^2.  \ee
Now, combining~\eqref{eq:ls_error}, \eqref{eq:mult_singulars}
and~\eqref{eq:desh1}, we get that
$$
\int_{H^{d+1}_\xcm} \MM^2_{d+1}(X) \di \mu^{d+1} (X) \leq \int_{H^{d}_\xcm}
\MM^2_{d}(\tilde{X}) \di \mu^{d} (\tilde{X}) \cdot
e_2^2(\mu,d),$$
that is,
\be \label{eq:cDnn2_final} \nonumber
\int_{H^{d+1}_\xcm} \cDnn2(X) \di \mu^{d+1} (X) \leq
e_2^2(\mu,d)\ee
and combining it with~\eqref{eq:fubini_cDnn} we conclude~\eqref{eq:deshpande}
and thus Theorem~\ref{thm:deshpande}.
\end{proof}



\section{Statistical Relevance of This Work}
\label{sec:stat}

\subsection{Application to the Problem of Clustering Subspaces}
\label{sec:subspace_clustering}

The identity of~\eqref{eq:elementary_identity} is useful for
clustering algorithms based on pairwise distances (see e.g.,
\cite{Brand_clust_unpublished}).
Similarly, the approximate identities of this paper are also useful
for clustering algorithms based on higher-order correlations~\cite{Govindu05,
Agarwal05, Shashua06, spectral_applied, spectral_theory,
Arias-Castro08Surfaces}. The latter algorithms are designed to cluster intersecting subspaces or manifolds, where
the former algorithms fail. For example, the Spectral Curvature Clustering
(SCC)~\cite{spectral_applied, spectral_theory} is an algorithm for
clustering $d$-dimensional affine subspaces. It assigns to any
$d+2$ data points, $x_1$, $\ldots$, $x_{d+2}$, the
affinity, $e^{-\cpols(x_1, \ldots, x_{d+2})/2 \sigma^2}$, where $\cpols$ is the polar GCN and
$\sigma$ is a positive tuning parameter that can be estimated from
the data. It then organizes these affinities
in a matrix whose spectral properties provide the clusters.
We remark that $\cpols$ was referred to in~\cite{spectral_applied, spectral_theory}
as curvature, instead of GCN,
and this resulted in the algorithm's name SCC.

The results of the current paper have been used to justify the SCC algorithm~\cite{spectral_theory}.
More precisely, \cite{spectral_theory} assumed data sampled from a
mixture of subspaces corrupted by sufficiently small noise and
showed that the underlying subspaces could be recovered with
sufficiently large probability and small error. This error was
controlled by  two terms: a sum of
within-clusters  errors scaled by $\sigma^2$ (where $\sigma$
is the tuning parameter used to define the affinities)
and between-clusters interaction. The control of the first term (involving within-clusters errors)
was established by some of the theory proved here. This theory is simpler and more general than the one
referred to in~\cite[Section~2.3]{spectral_theory}.

\subsection{From Estimates in Expectation to Estimates in High Probability}

We extend the comparisons of the two expected quantities (i.e, LS error, which is the expectation of $\dist^2(x,L)$, and the expectation of squared GCNs)
to comparisons of their estimators obtained by i.i.d.~samples from $\mu$.
That is, assume $N$ $H$-valued i.i.d.~random variables drawn from $\mu$, denoted by $\mathfrak{X}_1,\ldots,\mathfrak{X}_N$.
We can estimate the LS error and any of the integrals of squared GCNs (assume for simplicity $\cdls2$) as follows:
\be \label{eq:discrete_error}e_2^2(\mathfrak{X}_1,\ldots,\mathfrak{X}_N;d)=\frac{1}{N} \min_{L \in \agh}\sum_{i=1}^N \dist^2(\mathfrak{X}_{i},L)
\ee
and
\be \label{eq:discrete_curv}\cdls2(\mathfrak{X}_1,\ldots,\mathfrak{X}_N;d)=\frac{1}{N^{d+2}} \sum_{\latop{X=(\mathfrak{X}_{i_1},\ldots,\mathfrak{X}_{i_{d+2}})}{1\leq i_1,\ldots,i_{d+2}\leq N}}\cdls2(X).
\ee
The following theorem shows that these two quantities are comparable to each other with high probability of sampling.
\begin{theorem}
\label{thm:high_prob}
If $\mu$ is $d$-separated (for $\omega$ and $\epsilon$), $\mathfrak{X}_1,\ldots,\mathfrak{X}_N$ are $N$ $H$-valued i.i.d.~random variables drawn from $\mu$,
then for any $0<\delta<1$ and
$$\kappa = \frac{\delta}{(d+2) \cdot \diam(\mu)^2} \, \int_{H^{d+2}}\cdls2(X)\,\di\mu^{d+2}(X),$$
the following estimate holds with probability $1-2\cdot e^{-2\cdot N \cdot \kappa^2}:$
\be \label{eq:discrete_curv2}
\cdls2(\mathfrak{X}_1,\ldots,\mathfrak{X}_N;d) \leq e_2^2(\mathfrak{X}_1,\ldots,\mathfrak{X}_N;d) \leq \frac{1+\delta}{1-\delta}\cdot\frac{1}{\omega^2\cdot\epsilon^{d+1}}\cdot{\cdls2(\mathfrak{X}_1,\ldots,\mathfrak{X}_N;d)}.
\ee
Moreover, for any $0<\delta<\eps$ the following estimate holds with probability $1 - (d+1) \cdot e^{-2N\delta^2}:$
\be \label{eq:discrete_curv3}
\cdls2(\mathfrak{X}_1,\ldots,\mathfrak{X}_N;d) \leq e_2^2(\mathfrak{X}_1,\ldots,\mathfrak{X}_N;d) \leq \frac{1}{\omega^2\cdot(\epsilon-\delta)^{d+1}}
\cdot \cdls2(\mathfrak{X}_1,\ldots,\mathfrak{X}_N;d).
\ee
\end{theorem}

\begin{proof}
The LHS inequality of both~\eqref{eq:discrete_curv2} and~\eqref{eq:discrete_curv3} is proved identically to~\eqref{eq:prop_main-2} and in fact is a deterministic inequality.

We first verify the RHS inequality of~\eqref{eq:discrete_curv2} by estimating with probability the integral quantities by their discrete counterparts (via concentration inequalities). In order to estimate the integral of $\cdls2$ by $\cdls2(\mathfrak{X}_1,\ldots,\mathfrak{X}_N;d)$ we fix $1 \leq i\leq  N$ and note that the number of additive terms in
$\cdls2(\mathfrak{X}_1,\ldots,\mathfrak{X}_N;d)$
that contain $\mathfrak{X}_i$ is $(d+2)\cdot\perm{N-1}{d+1}$,
where $\perm{N-1}{d+1}$ denotes the permutations of $d+1$ elements out of $N-1$. Moreover, each of these terms is between 0 and $\diam(\mu)^2/N^{d+2}$. Consequently,
\be\label{eq:mcdiarmid}\sup_{\mathfrak{X}_1,\ldots,\mathfrak{X}_N,\widehat{\mathfrak{X}}_i}
|\cdls2(\mathfrak{X}_1,\ldots,\mathfrak{X}_i,
\ldots,\mathfrak{X}_N;d)-
\cdls2(\mathfrak{X}_1,\ldots,\widehat{\mathfrak{X}}_i,
\ldots, \mathfrak{X}_N;d)| \leq (d+2)\cdot \diam(\mu)^2/N.\ee
Applying McDiarmid's inequality~\cite{McDiarmid89} with the underlying condition expressed in~\eqref{eq:mcdiarmid} we obtain that for any $\beta>0$:
\begin{align}
\mu^N\left(\int_{H^{d+2}}\cdls2(X)\di\mu^{d+2}(X) - \cdls2(\mathfrak{X}_1,\ldots,\mathfrak{X}_N;d) \geq
\beta\right)
& \leq e^{-2N\beta^2/((d+2)^2\cdot \diam(\mu)^4)}.
\label{eq:mcdiarmid0_2}\end{align}
Setting
\be \label{eq:delta_by_c} \beta=\delta \, \int_{H^{d+2}}\cdls2(X)\di\mu^{d+2}(X),\ee
we rewrite~\eqref{eq:mcdiarmid0_2} as follows:
\begin{align}
\mu^N\left(\cdls2(\mathfrak{X}_1,\ldots,\mathfrak{X}_N;d)\leq (1-\delta) \, \int_{H^{d+2}}\cdls2(X)\di\mu^{d+2}(X)\right)
& \leq e^{-2N\beta^2/((d+2)^2\cdot \diam(\mu)^4)}.
\label{eq:mcdiarmid2}\end{align}

In order to estimate $e_2(\mu,d)$ by $e_2(\mathfrak{X}_1,\ldots,\mathfrak{X}_N;d)$ we note that
\be
e_2^2(\mathfrak{X}_1,\ldots,\mathfrak{X}_N;d)\leq \frac{1}{N} \sum_{i=1}^N \dist^2(\mathfrak{X}_{i},\widehat{L}), \label{eq:bound_e}\ee
where $\widehat{L}$ is a fixed LS $d$-flat for $\mu$.
Applying Hoeffding's inequality to the function on the RHS of~\eqref{eq:bound_e}, we obtain that
\begin{align}
\mu^N\left(\frac{1}{N} \sum_{i=1}^N \dist^2(\mathfrak{X}_{i},\widehat{L})-e_2^2(\mu,d)\geq
\beta\right)
& \leq e^{-2N\beta^2/\diam(\mu)^4}.
\label{eq:hoeffding_0}
\end{align}
By further use of~\eqref{eq:prop_main-2} and~\eqref{eq:delta_by_c}, we reduce~\eqref{eq:hoeffding_0} to the following probabilistic inequality:
\begin{align}
\mu^N\left( \frac{1}{N} \sum_{i=1}^N \dist^2(\mathfrak{X}_{i},\widehat{L}) \geq (1+\delta) \cdot e_2^2(\mu,d) \right)
& \leq e^{-2N\beta^2/\diam(\mu)^4}.
\label{eq:hoeffding}
\end{align}
The RHS inequality of~\eqref{eq:discrete_curv2} thus follows from the combination of~\eqref{eq:thm_main-1} \eqref{eq:mcdiarmid2}, \eqref{eq:bound_e} and~\eqref{eq:hoeffding}.

Next, we prove the RHS inequality of~\eqref{eq:discrete_curv3} by showing that $d$-separation of $\mu$ is maintained (for $\omega$ and $\eps'$, where $\eps'<\eps$) with overwhelming probability by i.i.d.~random variables sampled from $\mu$. We arbitrarily fix $j=1,\ldots,d+1$ and $i=1,\ldots,N$ and form the random variable $\mathfrak{I}_{i,j}$ by the formula:
$$\mathfrak{I}_{i,j}(x) = I_{\mathfrak{X}_i \in V_j}(x) \ \text{ for all } x \in H,$$
where $I$ is indicator function and $\{V_j\}_{j=1}^{d+1}$ are the sets used in defining the $d$-separation of $\mu$.
We note that
$$\int \mathfrak{I}_{i,j}(x)\di \mu(x) = \mu(V_j) > \eps.$$
Combining this observation with Hoeffding's inequality we obtain that
\begin{align*}
\mu^N\left(-\sum_{i=1}^N \mathfrak{I}_{i,j}/N+\eps\geq
\delta\right)
\leq
\mu^N\left(-\sum_{i=1}^N \mathfrak{I}_{i,j}/N+\mu(V_j)\geq
\delta\right)
& \leq e^{-2N\delta^2}.
\end{align*}
Consequently,
\begin{align*}
\mu^N\left(\bigcap_{j=1}^{d+1} \left(\sum_{i=1}^N \mathfrak{I}_{i,j}/N > \eps - \delta \right) \right)
& \geq 1 - (d+1) \cdot e^{-2N\delta^2}.
\end{align*}
That is, with probability $1 - (d+1) \cdot e^{-2N\delta^2}$ the  empirical measure $\mu_N(A)=\sum_{i=1}^N I_A(\mathfrak{X}_i)/N$ is $d$-separated for the parameters $\omega$ and $\eps-\delta$ and  the same sets $\{U_j\}_{j=1}^{d+1}$, $\{V_j\}_{j=1}^{d+1}$ .
For each such instance of $d$-separation of the empirical measure, we apply  Theorem~\ref{thm:main-1} to $\mu_N$.  That is, for a fixed sample $\mathfrak{X}_1,\ldots,\mathfrak{X}_N$ whose empirical measure $\mu_N$ is $d$-separated with these sets and constants $\omega$ and $\eps-\delta$, we have the following inequality which is simply Theorem~\ref{thm:main-1} applied to $\mu_N$:
\be \nonumber \label{eq:discrete_curv_for_samples}
e_2^2(\mathfrak{X}_1,\ldots,\mathfrak{X}_N;d) \leq \frac{1}{\omega^2\cdot(\eps-\delta)^{d+1}}
\cdot \cdls2(\mathfrak{X}_1,\ldots,\mathfrak{X}_N;d).
\ee
This inequality holds for all samples with probability $1 - (d+1) \cdot e^{-2N\delta^2}$ and the RHS inequality of~\eqref{eq:discrete_curv3} is thus concluded.
\end{proof}

\section{Discussion}
\label{sec:discuss}
We presented examples of $d$-dimensional geometric condition numbers whose integrals are comparable to the $d$-dimensional least squares error for certain classes of measures. We related these results to the problem of clustering subspaces and to volume-based sampling for Monte-Carlo SVD.
We discuss here further implications and open directions.

\subsection{Comparisons of $L_p$ Errors}

For simplicity we only discussed LS errors, i.e.,
$L_2$ errors. Nevertheless, $L_p$ errors for $1 \leq p < \infty$ can also
be estimated using $p$-th powers of the GCNs.

\subsection{Approximate Identities for Singular Values}
Some of the approximate identities established in this paper can be translated to
approximate identities involving singular values of certain operators. We
exemplify this claim for the data-to-features operator as follows.
\begin{theorem}\label{thm:sin_val_compare}
If $\mu$ is centrally $d$-separated (for $\omega$ and $\epsilon$) with compact support
and $\{\sigma_i\}_{i \in \nats}$
are the singular values of the data-to-features operator, then
\be \label{eq:sin_val_compare} \omega^2 \cdot \epsilon^{d} \,
\sum_{j=d+1}^\infty \sigma_j^2
\leq  \frac{\sum_{1 \leq t_1 < \ldots < t_{d+1}}
\sigma^2_{t_1} \cdots \sigma^2_{t_{d+1}} }{\sum_{1 \leq t_1 < \ldots
< t_{d}} \sigma^2_{t_1} \cdots \sigma^2_{t_{d}}  } \leq
\sum_{j=d+1}^\infty \sigma_j^2.  \ee
\end{theorem}
We note that the inequality on the RHS of~\eqref{eq:sin_val_compare}
is trivial for any set of numbers $\{\sigma_i\}_{i \in \nats}$.
The LHS comparability is an immediate corollary of Theorem~\ref{thm:deshpande}
(in view of~\eqref{eq:def_cdnn}-\eqref{eq:mult_singulars}).

\subsection{More Robust Notion of $d$-Separation}

Our notion of $d$-separation is not sufficiently ``robust to outliers'' since it depends on $\diam(\mu)$.
Assume, e.g., a probability measure which is a mixture of one component supported in the unit ball and another component of an atomic measure
supported on an arbitrarily far point with a sufficiently small weight. The diameter of this measure is mainly determined by the outlying atomic measure. However, for $X\in H^d_\xcm$ and the GCN $\cdes(X)$ (or $X\in H^{d+1}_\xcm$ and the GCN $\cDn(X)$) we can weaken the effect of outliers by replacing the condition
$\MM_d(X) \geq \omega\cdot \diam(\mu)^d$ with
\be
\MM^2_d(X) \geq \omega
\int_{H^d_\xcm}\MM^2_{d}(\tilde{Y})\di \mu^{d} (\tilde{Y})\,.
\ee

\subsection{On $d$-Separation w.r.t.~$(d+1)$-Simplices and Its Implications}\label{app:strong_d_sep}
A different notion of $d$-separation was previously used in the setting of $d$-regular measures on $H$~\cite{Leger99, LW-part2}.
It is based on $d$-separation of $(d+1)$-simplices (instead of $d$-simplices).
We adapt this notion to the current setting and explain its relation with $d$-separation defined here, we also describe its implications.

We say that a $(d+1)$-simplex  $X=(x_0,\ldots,x_{d+}) \in
\supp(\mu)^{d+2}$ is $d$-separated (for $\omega$) if all of its faces are $d$-separated as $d$-simplices (for $\omega$).
That is,
\begin{equation}\min_{0\leq i\leq d+1}\MM_d(X(i))\geq\omega\cdot
\diam(\mu)^d.\end{equation}
We say that $\mu$ is $d$-separated w.r.t.~$(d+1)$-simplices  (with positive constants
$\omega$, $\epsilon$ and $\tau$) if there exist sets $ V_i\subseteq
U_i \subseteq \supp(\mu)$, $0\leq i\leq d+1,$ such that for each
$0\leq i\leq d+1$:
\begin{enumerate}
\item \label{enum:d_sep1}$\mu( V_i
)\geq\epsilon$.
\item \label{enum:d_sep2}$\dist_\mu( V_i, U_i^c) := \inf_{\substack{x\in V_i\cap\Supp\\y\in U_i^c\cap\Supp}}\|x-y\|\geq\tau \cdot \diam(\mu)$.
\item \label{enum:d_sep3}
$\prod_{i=0}^{d+1}
 U_i\subseteq \left\{X\in \supp(\mu)^{d+2}:\min_{0\leq i\leq
d+1}\MM_d(X(i))\geq\omega\cdot \diam(\mu)^d\right\}$ .
\end{enumerate}

In view of Lemma~\ref{lem:either-or} and its proof $d$-separation is almost identical to $d$-separation w.r.t.~$(d+1)$-simplices. The typical example of a $d$-separated measure which is not $d$-separated with respect to $(d+1)$-simplices is a measure supported on $d+1$ atoms with positive $d$-volume. One can add another part of the support lying on a $(d-1)$-flat containing $d$ of these atoms and provide this way additional examples.

Nevertheless, the extra care taken in defining $d$-separation w.r.t.~$d$-simplices is necessary in formulating the following stronger version of Theorem~\ref{thm:main-1}, which  restricts the integral of $\cvolker(X)$ to the following set of simplices with sufficiently large edge lengths (with respect to $\tau$):
\begin{equation}\nonumber \label{equation:large-edges}LE_{\tau}(\mu)=\left\{X\in \supp(\mu)^{d+2}:\min(X)\geq\tau\cdot
\diam(\mu)\right\}.\end{equation}
\begin{theorem}\label{thm:main_modified}If $\mu$ is $d$-separated (for $\omega$, $\epsilon$ and $\tau$) w.r.t.~$(d+1)$-simplices,
then
\be \nonumber \label{eq:theorem_main}e_2^2(\mu,d)\leq
\frac{4}{\omega^2\cdot\epsilon^{d+1}} \left(1+4\cdot(d+1)^2+\frac{4\cdot(d+1)}
{\omega^2\cdot\epsilon} \right)
\, \int_{LE_\tau(\mu)}\cvolker(X)\di\mu^{d+2}(X).\ee
\end{theorem}
The proof of this theorem follows the one of~\cite[Theorem~1.1]{LW-part2}.
This type of control was necessary in~\cite{Leger99, LW-part2} since singular curvature functions were used instead of GCNs and they
had to be further integrated along various ``scales'' $t$ w.r.t.~the measure $\di t/t$.
Clearly, it is not necessary in the current context.

\subsection{Extension to Metric Spaces}
It will be interesting to extend some of our results to metric spaces. In particular, by choosing appropriate metric GCNs one can obtain a corresponding notion of an approximate best-fit subspace. This task is considered in~\cite{LZ12} for the purpose of clustering $d$-dimensional smooth structures in metric spaces.

\bibliographystyle{plain}
\bibliography{refs_4_04_11}

\end{document}